\documentclass[12pt,reqno]{amsart}
\usepackage{fullpage}
\usepackage[english]{babel}
\usepackage{csquotes}
\usepackage{enumitem}
\usepackage{bbm}
\usepackage[
    pdfusetitle,    
    pdfauthor={Hui-Hsiung Kuo, Pujan Shrestha, Sudip Sinha, Padmanabhan Sundar},    
    pdfsubject={anticipating stochastic calculus},    
    pdfkeywords={anticipating integral, stochastic integral, stochastic differential equation, near-martingale, optional stopping theorem, large deviation principles},    
    bookmarks=true,
    colorlinks=true,
    linkcolor=Navy,
    urlcolor=IndianRed,
    citecolor=DarkSlateBlue,
]{hyperref}

\usepackage[svgnames]{xcolor}
\usepackage{tikz}
\usetikzlibrary{decorations.markings}
\usetikzlibrary{tikzmark,arrows.meta}
\usetikzlibrary{patterns}
\usetikzlibrary{shapes.geometric}
\usepackage{cancel}

\usepackage[
    style=numeric,
    sorting=nyt,
    backend=biber,
    maxnames=4,
]{biblatex}
\usepackage[noabbrev]{cleveref}    

\definecolor{goldenpoppy}{rgb}{0.99, 0.76, 0.0}
\definecolor{lavenderblue}{rgb}{0.8, 0.8, 1.0}
\definecolor{bazaar}{rgb}{0.6, 0.47, 0.48}
\definecolor{byzantium}{rgb}{0.44, 0.16, 0.39}
\definecolor{electricpurple}{rgb}{0.75, 0.0, 1.0}
\definecolor{deepskyblue}{rgb}{0.0, 0.75, 1.0}
\definecolor{hanblue}{rgb}{0.27, 0.42, 0.81}
\definecolor{internationalkleinblue}{rgb}{0.0, 0.18, 0.65}

\newcommand{\floor}[1]{\ensuremath{\left\lfloor #1 \right\rfloor}}

\newcommand{\abs}[1]{\ensuremath{\left\lvert #1 \right\rvert}}
\newcommand{\norm}[1]{\ensuremath{\left\lVert #1 \right\rVert}}
\newcommand*{\inv}[1]{\ensuremath{{#1}^{-1}}}
\newcommand*{\E}{\ensuremath{\operatorname{\mathbb{E}}\!}}
\renewcommand*{\Pr}{\operatorname{\mathbb{P}}\!}
\newcommand*{\given}{\ensuremath{\; \middle| \;}}

\newcommand{\R}{\ensuremath{\mathbb{R}}}

\newcommand*{\dif}{\,\ensuremath{d}}

\newcommand*{\Del}{\,\ensuremath{\Delta}}

\newcommand{\br}[1]{\ensuremath{\left( #1 \right)}}
\newcommand{\bc}[1]{\ensuremath{\left\{ #1 \right\}}}
\newcommand{\bs}[1]{\ensuremath{\left[ #1 \right]}}
\newcommand{\ba}[1]{\ensuremath{\left\langle #1 \right\rangle}}

\newcommand{\F}{\mathcal{F}}

\theoremstyle{plain}
\newtheorem{theorem}{Theorem}[section]
\newtheorem{lemma}[theorem]{Lemma}
\newtheorem{proposition}[theorem]{Proposition}
\newtheorem{corollary}[theorem]{Corollary}

\theoremstyle{definition}
\newtheorem{definition}[theorem]{Definition}

\theoremstyle{remark}
\newtheorem{remark}[theorem]{Remark}
\numberwithin{equation}{section}

\begin{document}

\title[Near-martingales and anticipating linear SDEs]{On near-martingales and a class of anticipating linear SDEs}

\author{Hui-Hsiung Kuo}
\address{Hui-Hsiung Kuo: Department of Mathematics, Louisiana State University, Baton Rouge, LA 70803, USA}
\email{kuo@math.lsu.edu}

\author{Pujan Shrestha*}
\thanks{* Corresponding author}
\address{Pujan Shrestha: Department of Mathematics, Louisiana State University, Baton Rouge, LA 70803, USA}
\email{pujanshrestha57@gmail.com}

\author{Sudip Sinha}
\address{Sudip Sinha: Department of Mathematics, Louisiana State University, Baton Rouge, LA 70803, USA}
\email{sudipsinha@protonmail.com}
\urladdr{https://sites.google.com/view/sudip-sinha}

\author{Padmanabhan Sundar}
\address{Padmanabhan Sundar: Department of Mathematics, Louisiana State University, Baton Rouge, LA 70803, USA}
\email{psundar@lsu.edu}



\begin{abstract}
    The primary goal of this paper is to prove a near-martingale optional stopping theorem and establish solvability and large deviations for a class of anticipating linear stochastic differential equations. We prove the existence and uniqueness of solutions using two approaches: (1) Ayed--Kuo differential formula using an ansatz, and (2) a novel braiding technique by interpreting the integral in the Skorokhod sense. We establish a Freidlin--Wentzell type large deviations result for solution of such equations.
\end{abstract}

\maketitle

\section{Introduction}
Anticipating stochastic calculus has been an active and important research area for several years, and lies at the intersection of probability theory and infinite-dimensional analysis. Enlargement of filtration, Malliavin calculus, and white noise theory provide three distinct methodologies to incorporate anticipation (of future) into classical It\^o theory of stochastic integration and differential equations.

It is to the credit of It\^o who constructed an anticipating stochastic integral in 1976\cite{Ito1978}, and laid the foundation for the idea of enlargement of the underlying filtration. Ever since, the method was embraced by several researchers that led to many important works (see articles in \cite{JeulinYor1985}). The advent of an integral invented by Skorokhod resulted in an impressive edifice built by Malliavin on stochastic calculus of variations in order to prove H\"ormander's hypoellipticity result by stochastic analysis. Malliavin calculus provided a natural basis for the development and study of anticipative stochastic analysis and differential equations. Around the same time, a systematic study of Hida distributions gave rise to white noise theory and a general framework for stochastic calculus.

Malliavin calculus and white noise theory have  vast applicability to the theory of stochastic differential equations with anticipation. However, the results obtained by these theories are primarily abstract though general. A more tractable theory was envisaged by Kuo based on a concrete stochastic integral known as the Ayed--Kuo integral\cite{AyedKuo2008}. Under less generality, the latter allows one to obtain results under easily understood, verifiable hypotheses.

In this article, we prove some results about stopped near-martingales, which are generalizations of martingales. We then study existence, uniqueness and large deviation principle for linear stochastic differential equations with anticipating initial conditions and drifts. While we rely mostly on the Ayed--Kuo formalism, other theories are minimally used either out of necessity, or to compare and contrast the conclusions of certain results.

The structure of the paper is as follows. In \cref{sec:Ayed-Kuo}, we introduce the Ayed--Kuo integral. In \cref{sec:near-maringales}, study near-martingales. We show that Ayed--Kuo integrals are near-martingales. We also show that stopped near-martingales are near-martingales, and prove an optional stopping theorem for near-submartingales. In \cref{sec:SDE_solution}, we study methods for solving anticipating linear stochastic differential equations by interpreting the anticipating stochastic integral from two perspectives. For the Ayed--Kuo formulation, we use the differential formula and an ansatz to derive the solution. For the Skorokhod interpretation, we introduce a novel braiding technique inspired by Trotter's product formula\cite{Trotter1959}. We show that the solutions coincide when the assumptions are identical. Finally, in \cref{sec:LDP}, we derive large deviation principles for the solutions of the class of anticipating linear stochastic differential equations studied in \cref{sec:SDE_solution}. In this paper, we assume \( t \in [0, 1] \), unless specified otherwise.

\section{The Ayed--Kuo anticipating stochastic calculus}  \label{sec:Ayed-Kuo}

Before we define the Ayed--Kuo integral, we need to define instantly independent processes. A stochastic process \( \phi(t) \) is called \emph{instantly independent} with respect to \( \bc{\F_t} \) if for each \( t \in [0, 1] \), the random variable \( \phi(t) \) and the \( \sigma \)-field \( \F_t \) are independent. Instantly independent processes are the counterpart of adapted processes in this theory.

We refer to \cite[section 2]{HwangKuoSaitoZhai2016} for a detailed definition of the anticipating stochastic integral. In what follows, we highlight the crucial steps in the definition in a concise manner.
\begin{definition}[{\cite[definition 2.3]{HwangKuoSaitoZhai2016}}]
    The anticipating integral is defined in following three steps:
    \begin{enumerate}
        \item  Suppose \( f(t) \) is an \( \F_t \)-adapted continuous stochastic process and \( \phi(t) \) is a continuous stochastic processes that is instantly independent with respect to \{\( \F_t \)\}. Then the stochastic integral of \( \Phi(t) = f(t) \phi(t) \) is defined by
        \begin{equation*}
            \int_0^1 f(t) \phi(t) ~ d W_t = \lim_{\norm{\Delta_n} \to 0} \sum_{j = 1}^n f(t_{j-1}) ~  \phi(t_j) ~ (W_{t_j} - W_{t_{j-1}}) ,
        \end{equation*}
        provided that the limit exists in probability.
    
        \item  For any stochastic process of the form \( \Phi(t) = \sum_{i = 1}^n f_i(t) \phi_i(t) \), where \( f_i(t) \) and \( \phi_i(t) \) are given as in step (1), the stochastic integral is defined by
        \begin{equation*}
            \int_0^1 \Phi(t) ~ d W_t = \sum_{i = 1}^n \int_0^1 f_i(t) ~ \phi_i(t) ~ d W_t .
        \end{equation*}
    
        \item  Let \( \Phi(t) \) be a stochastic process such that there is a sequence \( \br{\Phi_n(t)}_{n = 1}^\infty \) of stochastic processes of the form in step (2) satisfying
        \begin{enumerate}
            \item  \( \int_0^1 \abs{\Phi_n(t) - \Phi(t)}^2 ~ d t \to 0 \) almost surely as \( n \to \infty \), and
            \item  \( \int_0^1 \Phi_n(t) ~ d W_t \) converges in probability as \( n \to \infty \).
        \end{enumerate}
        Then the stochastic integral of \( \Phi(t) \) is defined by
        \begin{equation*}
            \int_0^1 \Phi(t) ~ d W_t = \lim_{n \to \infty}  \int_0^1 \Phi_n(t) ~ d W_t \quad \text{in probability.}
        \end{equation*}
    \end{enumerate}
\end{definition}

This integral is well defined, as demonstrated in  \cite[lemma 2.1]{HwangKuoSaitoZhai2016}. In order to use the definition of the integral, we first need to decompose the integrand into a sum of products of adapted and instantly independent parts. The main idea is to then use the left-endpoints of subintervals to evaluate the adapted parts and the right-endpoints of subintervals to evaluate the instantly independent parts.

\section{Near-martingales}  \label{sec:near-maringales}

\subsection{Near-martingale property of the Ayed--Kuo integral}
Martingales are an extremely important class of processes that are used to model fair games, and hence find applications not only in probability theory, but also in mathematical finance and numerous other fields. It\^o's integrals are essentially continuous martingale transforms, and therefore retain the martingale nature of the integrator. Since the Ayed--Kuo integral is an extension of the It\^o integral, it is natural to ask if Ayed--Kuo integrals are martingale. Unfortunately, they are not. However, we have a very similar property, which gives rise to the idea of \emph{near-martingales}.

\begin{definition}
    An integrable stochastic process \( N_t \) is called a \emph{near-submartingale} with respect to the filtration \( \bc{\F_t} \) if for any $ s \leq t $, we have \( \E\br{N_t - N_s \given \F_s} \geq 0 \) almost surely. It is called a \emph{near-martingale} if \( \E\br{N_t - N_s \given \F_s} = 0 \) almost surely.
\end{definition}

The following result links martingales and near-martingales. In particular, it says that conditioned near-martingales are martingales.
\begin{theorem}[{\cite[theorem 2.11]{HwangKuoSaitoZhai2017}}]
    Let \( N_t \) be an integrable stochastic process and let \( M_t = \E\br{N_t \given \F_t} \). Then \( N_t \) is a near-martingale if and only if \( M_t \) is a martingale.
\end{theorem}

Ayed--Kuo integrals are near-martingales, as stated by this theorem.
\begin{theorem}  \label{thm:Ayed-Kuo_integral_near-martingale}
    Let $\Theta(x,y)$ be a function that is continuous in both variables such that the stochastic integral,
    \begin{align*}
        N_t = \int_a^t \Theta( W_s , W_b - W_s) \,dW_s, \qquad a \leq t \leq b,
    \end{align*}
    exists and $\mathbb{E}\vert\, N_t\, \vert < \infty$ for each $t$ in $[0, 1]$. Furthermore, assume that the family of partial sums
    \begin{align*}
        \sum_{i=1}^n \Theta(W_{t_i}, W_1 -W_{t_{i-1}})     \br{W_{t_i} - W_{t_{i-1}} }
    \end{align*}
    are uniformly integrable. Then $N_t , \: a \leq t \leq b $, is a near-martingale with respect to the filtration generated by Brownian motion given by \( \bc{\F_t} \)
\end{theorem}

\begin{proof}
    Let \( s \leq t \) and consider a partition, \( \Delta_n \), of \( [s, t] \) with \( t_{0} = s \) and \( t_{n} = t \). The definition of the Ayed--Kuo stochastic integral in conjunction with the uniform integrability condition on the partial sums implies
    \begin{align}
        \mathbb{E} \left[ N_t - N_s \, \vert \, \F_s \right] =&  \mathbb{E} \left[ \int_s^t \Theta( W_v , W_b - W_v) \,dW_v \, \vert \, \F_s \right] \nonumber \\
        =&   \mathbb{E} \left[ \lim_{n \rightarrow \infty} \sum_{k=1}^n \Theta( W_{k-1} , W_b - W_k) \Delta W_k \, \vert \, \F_s \right] \nonumber \\
        =&   \lim_{n \rightarrow \infty} \sum_{k=1}^n \mathbb{E} \left[  \Theta( W_{k-1} , W_b - W_k) \Delta W_k \, \vert \, \F_s \right] . \label{eqn:nmglimitform}
    \end{align}
    
    Consider, \( \mathcal{H}^{(b)}_a = \sigma (\F_a \cup \mathcal{G}^{(b)}) \). Then \( \F_s \subseteq \F_{k-1} \subseteq \mathcal{H}^{(k)}_{k-1} \). Using this fact alongside the continuity of \( \Theta \) in both variables, we have that \( \Theta(W_{k-1} , W_b - W_k) \) is \( \mathcal{H}^{(k)}_{k-1} \)-measurable. Furthermore, via the independence of the Brownian increments, \( \Delta W_k \) is independent of \( \mathcal{H}^{(k)}_{k-1} \). Thus,
    \begin{align*}
        & \mathbb{E} \left[  \Theta( W_{k-1} , W_b - W_k) \Delta W_k \, \vert \, \F_s \right] \\
        & = \mathbb{E} \left[ \mathbb{E} \left[  \Theta( W_{k-1} , W_b - W_k) \Delta W_k \, \vert \, \mathcal{H}^{(k)}_{k-1} \right]\vert \, \F_s \right]\\
        & = \mathbb{E} \left[  \Theta( W_{k-1} , W_b - W_k) \mathbb{E} \left[   \Delta W_k \right]\vert \, \F_s \right] \\
        & = 0 .
    \end{align*}

    Using this result for each \( k \) in equation \eqref{eqn:nmglimitform}, we have \( \mathbb{E} \left[ N_t - N_s \, \vert \, \F_s \right]  = 0 \), and so \( N_t \) is a near-martingale.
    
    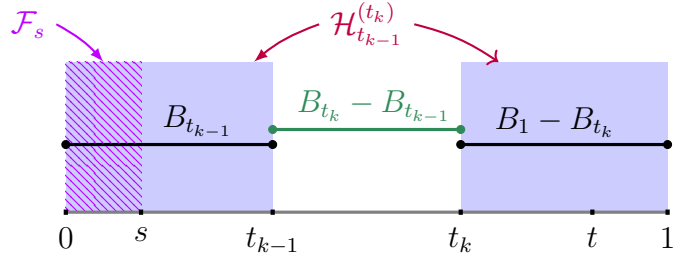
\begin{figure}[ht]
        \centering
        \begin{tikzpicture}
        \centering
            \fill [lavenderblue] (-2,0) rectangle (.75,2);
            \fill [pattern=north west lines, pattern color=electricpurple] (-2,0) rectangle (-1,2);
            \fill [lavenderblue] (6,0) rectangle (3.25,2);

            \node (separation) at (2, 2.5) [purple] {\( \mathcal{H}_{t_{k-1}}^{(t_k)} \)};
            \draw [-latex, thick, color=purple] (separation.west) to [bend right=10] (  0.5,2);
            \draw [->, thick, color=purple] (separation.east) to [bend left=10 ] (3.75,2);

            \node (separation) at (-2.5, 2.5) [electricpurple] {\( \F_s \)};
            \draw [-latex, thick, color=electricpurple] (separation.east) to [bend left=10] (  -1.5,2);

            \draw [-, very thick, color=gray] (-2,0) -- (6,0) node[below] {\(  \)};
            \foreach \x/\xtext in {-2/0, -1/s, 0.75/t_{k-1},3.25/t_{k},5/t, 6/1}
                \draw [ultra thick] (\x cm,1pt) -- (\x cm,-1pt) node[anchor=north] {$\xtext$};

            \draw [very thick, color=Black] (-2,.9) -- (0.75,.9);
            \draw[color=Black] (-.25,1.2) node {$B_{t_{k-1}}$};
            \draw [fill=Black] (-2,.9) circle (1.5pt);
            \draw [fill=Black] (0.75,.9) circle (1.5pt);

            \draw [very thick, color=Black] (3.25,.9) -- (6,.9);
            \draw[color=Black] (4.5,1.2) node {$B_1 - B_{t_k}$};
            \draw [fill=Black] (3.25,.9) circle (1.5pt);
            \draw [fill=Black] (6,.9) circle (1.5pt);

            \draw [very thick, color=SeaGreen] (.75,1.1) -- (3.25,1.1);
            \draw[color=SeaGreen] (2.1,1.4) node {$ B_{t_k} - B_{t_{k-1}}$};
            \draw [fill=SeaGreen, draw = SeaGreen] (.75,1.1) circle (1.5pt);
            \draw [fill=SeaGreen, draw = SeaGreen] (3.25,1.1) circle (1.5pt);

        \end{tikzpicture}
        \caption{A \( t \)-dependence plot of the disjoint increments of $ W $. The shaded regions represents the forward  and separation $\sigma$-field.}
    \end{figure}
\end{proof}

\subsection{Stopped near-martingales}

In this section, we show that stopped near-martingales are near-martingales. We also generalize Doob's optional stopping theorem for near-martingales.

\begin{definition}
    Let \( \br{A_n}_{n = 0}^\infty \) be an adapted process and \( \br{X_n}_{n = 0}^\infty \) a discrete time near-submartingale. Then the processes \( \br{Y_n}_{n = 0}^\infty \), where \( Y_0 = 0 \) and
    \begin{equation*}
        Y_n = (A \bullet X)_n = \sum_{i = 1}^n A_{n-1} (X_n - X_{n-1})
    \end{equation*}
    is called the \emph{near-martingale transform}\index{near-martingale transform} of \( X \) by \( A \).
\end{definition}

Near-martingale transforms retain the near-martingale property, as is shown in the following result.
\begin{proposition}  \label{thm:near-martingale_transform}
    \begin{enumerate}
        \item \label{itm:near-martingale_transform_bounded_positive}  If \( X \) is a near-submartingale and \( A \) is a bounded non-negative adapted process, then \( (A \bullet X) \) is a near-submartingale.
        \item \label{itm:near-martingale_transform_bounded} If \( X \) is a near-martingale and \( A \) is a bounded adapted process, then \( (A \bullet X) \) is a near-martingale.
        \item  If \( X \) and \( A \) are both square integrable, then we do not require the boundedness condition in \cref{itm:near-martingale_transform_bounded_positive,itm:near-martingale_transform_bounded}.
    \end{enumerate}
\end{proposition}
\begin{proof}
    We only prove \cref{itm:near-martingale_transform_bounded_positive} because the rest follow the same process. Let \( X \) be a near-submartingale and \( Y = (A \bullet X) \). Suppose \( n \) is an arbitrary time. Note that \( Y_n - Y_{n-1} = A_{n-1} (X_n - X_{n-1}) \), which is integrable since \( A \) is bounded. Using the adaptedness of \( A \), we get
    \begin{equation*}
        \E\br{Y_n - Y_{n-1} \given \F_{n-1}}
        =  \E\br{A_{n-1} (X_n - X_{n-1}) \given \F_{n-1}}
        =  A_{n-1} \E\br{X_n - X_{n-1} \given \F_{n-1}}
        \geq  0 ,
    \end{equation*}
    where the last inequality holds since \( A \) is non-negative.
\end{proof}

We now show that stopped near-submartingales are near-submartingales.
\begin{theorem}
    Suppose \( X \) is a discrete time near-submartingale and \( \tau \) a stopping time. Then the stopped process \( X^\tau \) defined by \( X^\tau_n = X_{\tau \wedge n} \) is a (discrete time ) near-submartingale.
\end{theorem}
\begin{proof}
    Let \( A_n = \mathbbm{1}_{\bc{n \leq \tau}} \), so the process \( A \) is bounded, non-negative, and adapted. Now, note that \( X^\tau_n - X_0 = X_{\tau \wedge n} - X_0 = (A \bullet X)_n \). Therefore, by \cref{thm:near-martingale_transform}, we get that \( X^\tau \) is a near-submartingale.
\end{proof}

Now, we show the equivalent result of Doob's optional stopping theorem for discrete time near-submartingales.
\begin{theorem}  \label{thm:optional_stopping_near-martingale_discrete_time}
    Let \( X \) be a discrete time near-submartingale. Suppose \( \sigma \) and \( \tau \) are two bounded stopping times with \( \sigma \leq \tau \). Then \( X_\sigma \) and \( X_\tau \) are integrable, and \( \E\br{X_\tau - X_\sigma \given \F_\sigma} \geq 0 \) almost surely.
\end{theorem}
\begin{proof}
    Since \( \sigma \) and \( \tau \) are bounded, there exists \( N < \infty \) such that \( \sigma \leq \tau \leq N \). Let \( Y \) be any near-submartingale. Clearly, \( Y_\sigma \) is integrable. Suppose \( B \in \F_\sigma \). Then for any \( n \leq N \), we have \( B \cap \bc{\sigma = n} \in \F_n \), and so
    \[ \int_{B \cap \bc{\sigma = n}} \br{Y_N - Y_\sigma} \dif \Pr  ~ =  \int_{B \cap \bc{\sigma = n}} \br{Y_N - Y_n} \dif \Pr  ~ \geq  0 . \]
    Summing over \( n \), we get \( \int_B \br{Y_N - Y_\sigma} \dif \Pr \geq 0 \), and so \( \E\br{Y_N - Y_\sigma \given \F_\sigma} \geq 0 \). Finally, let \( Y_n = X^\tau_n \) to get
    \[ \E\br{X^\tau_N - X^\tau_\sigma \given \F_\sigma}  =  \E\br{X_\tau - X_\sigma \given \F_\sigma}  \geq  0 . \]
\end{proof}

We need the following definition and lemma to prove the result in continuous time.
\begin{definition}
    Let \( \br{\F_n}_{n = 1}^\infty \) be a decreasing sequence of \( \sigma \)-algebras, and let \( X = \br{X_n}_{n = 1}^\infty \) be a stochastic process. Then the pair \( \br{X_n, \F_n}_{n = 1}^\infty \) is called a \emph{backward near-submartingale} if for every \( n \),
    \begin{enumerate}
        \item  \( X_n \) is integrable and \( \F_n \)-measurable, and
        \item  \( \E\br{X_n - X_{n+1} \given \F_{n+1}} \geq 0 \).
    \end{enumerate}
\end{definition}

\begin{lemma}  \label{thm:backward_near-submartingale_UI}
    Let \( \br{X_n, \F_n}_{n = 1}^\infty \) be a backward near-submartingale with \( \lim_{n \to \infty} \E\br{X_n} > -\infty \). If \( X \) is non-negative for every \( n \), then \( X \) is uniformly integrable.
\end{lemma}
\begin{proof}
    As \( n \nearrow \infty \), we have \( \E\br{X_n} \searrow \lim_{n \to \infty} \E\br{X_n} = \inf_n \E\br{X_n} > -\infty \). Fix \( \epsilon > 0 \). By the definition of infimum, there exists a \( K > 0 \) such that for any \( n \geq K \), we have \( \E\br{X_K} - \lim_{n \to \infty} \E\br{X_n} < \epsilon \).

    For any \( k \geq n \) and \( \delta > 0 \), we have
    \begin{equation*}
        \E\br{\abs{X_k} \mathbbm{1}_{\bc{\abs{X_k} > \delta}}}
        =  \E\br{X_k \mathbbm{1}_{\bc{X_k > \delta}}} + \E\br{ X_k \mathbbm{1}_{\bc{X_k \geq -\delta}}} - \E\br{X_k} .
    \end{equation*}
    Moreover, since \( X \) is a backward near-submartingale, \( \E\br{X_k \mathbbm{1}_{\bc{X_k \geq \delta}}} \leq \E\br{X_n \mathbbm{1}_{\bc{X_k \geq \delta}}} \). Therefore,
    \begin{align*}
        \E\br{\abs{X_k} \mathbbm{1}_{\bc{\abs{X_k} > \delta}}}
        & \leq  \E\br{X_n \mathbbm{1}_{\bc{X_k > \delta}}} + \E\br{X_n \mathbbm{1}_{\bc{X_k \geq -\delta}}} - \br{\E\br{X_n} - \epsilon}  \\
        & \leq  \E\br{\abs{X_n} \mathbbm{1}_{\bc{\abs{X_k} > \delta}}} + \epsilon .
    \end{align*}
    By Markov's inequality and the non-negativity of \( X \),
    \begin{equation*}
        \Pr\bc{\abs{X_k} > \delta}
        \leq  \frac1\delta \E\abs{X_k}
        =  \frac1\delta \E\br{X_k}
        \leq  \frac1\delta \E\br{X_1}
        \to 0
    \end{equation*}
    as \( \delta \to \infty \). This concludes the proof.
\end{proof}

We are now ready to prove the near-martingale optional stopping theorem in continuous time.
\begin{theorem}  \label{thm:optional_stopping_near-martingale_continuous}
    Let \( N \) be a near-submartingale with right-continuous sample paths. Suppose \( \sigma \) and \( \tau \) are two bounded stopping times with \( \sigma \leq \tau \). If \( N \) is either non-negative or uniformly integrable, then \( N_\sigma \) and \( N_\tau \) are integrable, and
    \[ \E\br{N_\tau - N_\sigma \given \F_\sigma} \geq 0 \text{ almost surely.} \]
\end{theorem}
\begin{proof}
    We use a discretization argument to prove the result. Let \( T > 0 \) be a bound for \( \tau \). For every \( n \in \mathbb{N} \), define the discretization function
    \begin{equation}
        f_n: [0, \infty) \to \bc{\frac{k}{n} : k = 0, \dotsc, n}: x \mapsto \frac{\floor{2^n x} + 1}{2^n} \wedge T ,
    \end{equation}
    and let \( \sigma_n = f_n(\sigma) \) and \( \tau_n = f_n(\tau) \).
    
    Now, for any \( n \) and \( t \),
    \begin{equation*}
        \bc{\tau_n \leq t}
        =  \bc{f_n(\tau) \in [0, t]}
        =  \bc{\tau \in f_n^{-1}[0, t]}
        =  \bc{\tau \in f_n^{-1} \bs{0, {\frac{\floor{2^n t}}{2^n}}}}
        \in  \F_{\frac{\floor{2^n t}}{2^n}} \subseteq \F_t ,
    \end{equation*}
    so \( \tau_n \) is a stopping time. Similarly, \( \sigma_n \) is a stopping time. Moreover, it can be easily seen that \( \sigma_n \leq \tau_n \) for every \( n \), and \( \sigma_n \searrow \sigma \) and \( \tau_n \searrow \tau \) as \( n \nearrow \infty \). Therefore, by \cref{thm:optional_stopping_near-martingale_discrete_time}, we get \( N_{\sigma_n} \) and \( N_{\tau_n} \) are integrable, and \( \E\br{N_{\tau_n} - N_{\sigma_n} \given \F_{\sigma_n}} \geq 0 \) almost surely. Furthermore, it is easy to see that \( \F_\sigma = \bigcap_{n = 1}^\infty \F_{\sigma_n} \subseteq \F_{\sigma_n} \) for any \( n \). Therefore, \( \E\br{N_{\tau_n} - N_{\sigma_n} \given \F_\sigma} \geq 0 \) almost surely for any \( n \).

    If \( N \) is non-negative, by construction, \( \br{N_{\sigma_n}, \F_{\sigma_n}}_{n = 1}^\infty \) is a backward near-submartingale such that \( N_{\sigma_n} \geq 0 \) for every \( n \). Therefore, \( \E\br{N_{\sigma_n}} \searrow \E\br{N_{\sigma}} > -\infty \) as \( n \nearrow \infty \). Using \cref{thm:backward_near-submartingale_UI}, \( \br{N_{\sigma_n}}_{n = 1}^\infty \) is uniformly integrable. Similarly, \( \br{N_{\tau_n}}_{n = 1}^\infty \) is also uniformly integrable. On the other hand, if \( N \) is uniformly integrable, this is trivial.

    Using the right continuity of \( N \) and the boundedness assumption of \( \sigma \) and \( \tau \), we get \( \lim_{n \to \infty} N_{\sigma_n} = N_\sigma \) and \( \lim_{n \to \infty} N_{\tau_n} = N_\tau \) almost surely. Furthermore, the uniform integrability of \( \br{N_{\sigma_n}}_{n = 1}^\infty \) and \( \br{N_{\tau_n}}_{n = 1}^\infty \) allows us to conclude that \( N_\sigma \) and \( N_\tau \) are integrable and that the convergence is also in \( L^1 \), giving us \( \E\br{N_\tau - N_\sigma \given \F_\sigma} \geq 0 \) almost surely.
\end{proof}

We highlight the special case of \cref{thm:optional_stopping_near-martingale_continuous}.
\begin{corollary}  \label{thm:optional_stopping_near-martingale_special}
    Let \( N \) be a non-negative near-martingale with right-continuous sample paths and \( \tau \) a bounded stopping time. Then \( N_\tau \) is integrable, and \( \E\br{N_\tau} = \E\br{N_0} \) almost surely.
\end{corollary}

\section{Anticipating linear stochastic differential equations}
\label{sec:SDE_solution}

In our previous works \cite{KuoSinhaZhai2018,KuoShresthaSinha2021conditional}, we studied linear stochastic differential equations with anticipating initial conditions, where the stochastic integral is in the Ayed--Kuo sense. In \cite{HwangKuoSaito2019}, the authors gave examples of linear stochastic differential equations with anticipating diffusion coefficient. In this paper, we focus on anticipating drift.

In particular, we shall be concerned about the solution of a class of anticipating linear stochastic differential equations of the form
\begin{equation}  \label{eqn:SDE_drift}
\left\{
\begin{aligned}
    dX_t  & =  \sigma_t X_t \, d W_t + f \Big(\int_0^1 \gamma_t \, d W_t \Big) X_t dt ,  \quad  t \in [0, 1] , \\
    \quad X_0  & =  \xi ,
\end{aligned}
\right.
\end{equation}
where \( W_t \) is a Brownian motion, \( f: \R \to \R \) is bounded function, \( \xi \) a random variable, and \( \sigma_t \) is an bounded adapted process such that all integrability conditions are satisfied. We choose this class because we want to study linear stochastic differential equations where the anticipation comes from the drift coefficient being Brownian functionals.

\subsection{The Ayed--Kuo sense}  \label{sec:SDE_drift_Ayed-Kuo}

We look at an extension of It\^o's formula that can account for instantly independent processes. Let \( X_t \) and \( Y^{(t)} \) be stochastic processes of the form
\begin{align}
	X_t  & =  X_a + \int_a^t g(s) \, d B(s) + \int_a^t h(s) \, d s ,  \label{eqn:differential-processes-adapted}  \\
	Y^{(t)}  & =  Y^{(b)} + \int_t^b \xi(s) \, d B(s) + \int_t^b \eta(s) \, d s , \label{eqn:differential-processes-independent}
\end{align}
where \( g(t), h(t) \) are adapted (so \( X_t \) is an It\^o process), and \( \xi(t), \eta(t) \) are instantly independent such that \( Y^{(t)} \) is also instantly independent.

\begin{theorem}[{\cite[theorem 3.2]{HwangKuoSaitoZhai2016}}] \label{thm:Ayed-Kuo_differential_formula}
	Suppose \( \{ X^{(i)}_t \}_{i = 1}^n \) and \( \{ Y_j^{(t)} \}_{j = 1}^m \) are stochastic processes of the form given by equations \eqref{eqn:differential-processes-adapted} and \eqref{eqn:differential-processes-independent}, respectively. Suppose \( \theta(t, x_1, \dotsc, x_n, y_1, \dotsc, y_m) \) is a real-valued function that is \( C^1 \) in \( t \) and \( C^2 \) in other variables. Then the stochastic differential of \( \theta(t, X^{(1)}_t, \dotsc, X^{(n)}_t, Y_1^{(t)}, \dotsc, Y_m^{(t)}) \) is given by
	\begin{align*}
		& d\theta(t, X^{(1)}_t, \dotsc, X^{(n)}_t, Y_1^{(t)}, \dotsc, Y_m^{(t)})  \\
		& = \theta_t \, dt
		+ \sum_{i = 1}^n \theta_{x_i} dX^{(i)}_t
		+ \sum_{j = 1}^m \theta_{y_j} dY_j^{(t)}  \\
		& \quad + \frac12 \sum_{i, k = 1}^n \theta_{x_i x_k} dX^{(i)}_t dX^{(k)}_t
		- \frac12 \sum_{j, l = 1}^m \theta_{y_j y_l} dY_j^{(t)} dY_l^{(t)}. \label{eqn:differential-formula}  
	\end{align*}
\end{theorem}

The above differential formula allows us to calculate the solutions of anticipating stochastic differential equations. We shall see two instances of its application in \cref{sec:SDE_drift_Ayed-Kuo}.

We apply the differential formula to derive a general result for existence of linear stochastic differential equations with anticipating coefficients.

\begin{theorem}  \label{thm:SDE_drift_Ayed--Kuo}
    Suppose \( \sigma \in L^2_\text{ad}([0, 1] \times \Omega) \), \( \gamma \in L^2[0, 1] \), and \( \xi \) be a random variable independent of the Wiener process \( W \). Moreover, suppose \( f \in C^2(\R) \) along with \( f, f', f'' \in L^1(\R) \). Then the solution of \cref{eqn:SDE_drift} in the Ayed--Kuo theory is given by
    \begin{equation}  \label{eqn:SDE_drift_solution_Ayed--Kuo}
        Z_t = \xi \exp\bs{ \int_0^t \sigma_s \dif W_s - \frac12 \int_0^t \sigma_s^2 \dif s + \int_0^t f\br{ \int_0^1 \gamma_u \dif W_u - \int_s^t \gamma_u ~ \sigma_u \dif u } \dif s } .
    \end{equation}
\end{theorem}

\begin{proof}
    We show that \cref{eqn:SDE_drift_solution_Ayed--Kuo} solves \cref{eqn:SDE_drift}. The initial condition is trivially verified.

    Note that \cref{eqn:SDE_drift_solution_Ayed--Kuo} can be written as
    \begin{align*}
        Z_t  =
        &  \xi \exp\left[ \int_0^t \sigma_s \dif W_s - \frac12 \int_0^t \sigma_s^2 \dif s \right.  \\
        &  \qquad \qquad  \left. + \int_0^t f\br{ \int_0^t \gamma_u \dif W_u + \int_t^1 \gamma_u \dif W_u - \int_s^t \gamma_u ~ \sigma_u \dif u } \dif s \right] .
    \end{align*}
    Motivated by this, we define
    \begin{equation*}
        \theta(t, x_1, x_2, y) =  \xi \exp\bs{ x_1 - \frac12 \int_0^t \sigma_s^2 \dif s + \int_0^t f\br{ x_2 + y - \int_s^t \gamma_u ~ \sigma_u \dif u } \dif s } .
    \end{equation*}
    Moreover, let
    \begin{align*}
                X^{(1)}_t  & =  \int_0^t \sigma_s  \dif W_s  &&  (\text{so } \dif X^{(1)}_t  =  \sigma_t  \dif W_t)  ,  \\
                X^{(2)}_t  & =  \int_0^t \gamma_s \dif W_s  &&  (\text{so } \dif X^{(2)}_t  =  \gamma_t \dif W_t)  ,  \\
        \text{ and }  Y^{(t)}  & =  \int_t^1 \gamma_s \dif W_s  &&  (\text{so }       \dif Y^{(t)}  = -\gamma_t \dif W_t) .
    \end{align*}
    Then we can write \( Z_t = \theta\br{t, X^{(1)}_t, X^{(2)}_t, Y^{(t)}} \).

    For conciseness, we denote \( F = f\br{\int_0^1 \gamma_t \dif W_t - \int_s^t \gamma_u ~ \sigma_u \dif u} \), and similarly the derivatives \( F' = f' \br{\int_0^1 \gamma_t \dif W_t - \int_s^t \gamma_u ~ \sigma_u \dif u} \) and \( F'' = f'' \br{\int_0^1 \gamma_t \dif W_t - \int_s^t \gamma_u ~ \sigma_u \dif u} \). Note that for the derivatives of \( \theta \), we have
    \begin{align*}
        \theta_{x_1}  & =  \theta_{x_1 x_1}  =  \theta , \\
        \theta_{x_2}  & =  \theta_{x_1 x_2}  =  \theta_y  =  \theta \cdot \int_0^t F' \dif s , \\
        \theta_{x_2 x_2}  & =  \theta_{y y}  =  \theta \cdot \br{\int_0^t F' \dif s}^2 + \theta \cdot \int_0^t F''\dif s , \text{ and} \\
        \theta_t  & =  -\frac12 \theta \sigma_t^2 + \theta f(x_2 + y) - \gamma_t \sigma_t \theta_y ,
    \end{align*}
    where we used the Leibniz integral rule and the second line for the last identity.

    Since \( \xi \) is independent of the Wiener process, by \cref{thm:Ayed-Kuo_differential_formula}, we get
    \begin{align*}
        \dif \theta
        & =  \theta_t \dif t + \theta_{x_1} \dif X^{(1)}_t + \theta_{x_2} \dif X^{(2)}_t + \theta_{y} \dif Y^{(t)}  \\
        &  \quad  + \frac12 \theta_{x_1 x_1} \br{\dif X^{(1)}_t}^2 + \frac12 \theta_{x_2 x_2} \br{\dif X^{(2)}_t}^2 + \theta_{x_1 x_2} \dif X^{(1)}_t \dif X^{(2)}_t - \frac12 \theta_{y y} (\dif Y^{(t)})^2 .
    \end{align*}
    Using the relationships between the derivatives of \( \theta \) and its differential form, we have
    \begin{align*}
        \dif \theta
        & =  \theta_t \dif t + \theta \sigma_t \dif W_t + \bcancel{\theta_y \gamma_t \dif W_t} - \bcancel{\theta \gamma_t \dif W_t}  \\
        & \quad + \frac12 \theta \sigma_t^2 \dif t + \cancel{\frac12 \theta_{y y} \gamma_t^2 \dif t} + \theta_y \gamma_t \sigma_t \dif t - \cancel{\frac12 \theta_{y y} \gamma_t^2 \dif t}  \\
        & = \br{\theta_t + \frac12 \theta \sigma_t^2 + \theta_y \gamma_t \sigma_t} \dif t + \theta \sigma_t \dif W_t .
    \end{align*}
    Now,
    \begin{equation*}
        \theta_t + \frac12 \theta \sigma_t^2 + \theta_y \gamma_t \sigma_t
        =  \br{ -\bcancel{\frac12 \theta \sigma_t^2} + \theta f(x_2 + y) - \cancel{\theta_y \gamma_t \sigma_t} } + \bcancel{\frac12 \theta \sigma_t^2} + \cancel{\theta_y \gamma_t \sigma_t}
        =  \theta f(x_2 + y) ,
    \end{equation*}
    and so
    \[ \dif \theta  =  f(x_2 + y) \theta \dif t + \sigma_t \theta \dif W_t . \]
    Since \( Z_t = \theta\br{t, X^{(1)}_t, X^{(2)}_t, Y^{(t)}} \), we get
    \[ \dif Z_t = f\br{\int_0^1 \gamma_s \dif W_s} Z_t \dif t + \sigma_t ~ Z_t \dif W_t , \]
    which is exactly \cref{eqn:SDE_drift}.
\end{proof}

\Cref{thm:Ayed-Kuo_differential_formula} is an indispensable tool for analyzing anticipating processes. We show another example by finding the stochastic differential equation corresponding to the square of the above solution.

\begin{theorem}  \label{thm:SDE_drift_solution_squared}
    Under the condition of \cref{thm:SDE_drift_Ayed--Kuo}, the stochastic differential equation
    \begin{equation*}
        \left\{
        \begin{aligned}
            \dif V_t
            & =  \bs{ \sigma_t^2 + f\br{\int_0^1 \gamma_s \dif W_s} + 2 \gamma_t ~ \sigma_t \int_0^t f'\br{\int_0^1 \gamma_u \dif W_u - \int_s^t \gamma_u ~ \sigma_u \dif u} \dif s } V_t \dif t  \\
            & \quad  + 2 \sigma_t ~ V_t \dif W_t , \\
            V_0  & =  \xi^2
        \end{aligned}
        \right.
    \end{equation*}
    is solved by \( Z_t^2 \), where \( Z \) is given by \cref{eqn:SDE_drift_solution_Ayed--Kuo}.
\end{theorem}

\begin{remark}
    An interesting feature is that the derivative of \( f \) appears in the stochastic differential equation.
\end{remark}

\begin{proof}
    We follow the exact same strategy as the proof of \cref{thm:SDE_drift_Ayed--Kuo}. The initial condition is trivially true. Let \( V_t = Z_t^2 \).

    Taking the square of both sides of \cref{eqn:SDE_drift_solution_Ayed--Kuo}, we get
    \begin{align*}
        V_t = \xi^2 \exp\bs{ \int_0^t 2 \sigma_s \dif W_s - \int_0^t \sigma_s^2 \dif s + \int_0^t 2 f\br{ \int_0^1 \gamma_u \dif W_u - \int_s^t \gamma_u ~ \sigma_u \dif u } \dif s }
    \end{align*}
    We have \( V_t = \theta\br{t, X^{(1)}_t, X^{(2)}_t, Y^{(t)}} \), where
    \begin{align*}
        \theta(t, x_1 , x_2 , y) = \xi^2 \exp\bs{ x_1  - \int_0^t \sigma_s^2 \dif s + \int_0^t 2 f\br{ x_2 + y - \int_s^t \gamma_u ~ \sigma_u \dif u } \dif s } ,
    \end{align*}
    and
    \begin{align*}
                X^{(1)}_t  & =  \int_0^t 2 \sigma_s \dif W_s  &&  (\text{so } \dif X^{(1)}_t  =  2 \sigma_t \dif W_t)  ,  \\
                X^{(2)}_t  & =  \int_0^t \gamma_s \dif W_s    &&  (\text{so } \dif X^{(2)}_t  =  \gamma_t \dif W_t)  ,  \\
        \text{ and }  Y^{(t)}  & =  \int_t^1 \gamma_s \dif W_s    &&  (\text{so }       \dif Y^{(t)}  = -\gamma_t \dif W_t) .
    \end{align*}
    As before, writing \( F = f\br{\int_0^1 \gamma_t \dif W_t - \int_s^t \gamma_u ~ \sigma_u \dif u} \), \( F' = f'\br{\int_0^1 \gamma_t \dif W_t - \int_s^t \gamma_u ~ \sigma_u \dif u} \), and \( F'' = f''\br{\int_0^1 \gamma_t \dif W_t - \int_s^t \gamma_u ~ \sigma_u \dif u} \), we get
    \begin{align*}
        \theta_{x_1}  & =  \theta_{x_1 x_1}  =  \theta , \\
        \theta_{x_2}  & =  \theta_{x_1 x_2}  =  \theta_y  =  2 \theta \cdot \int_0^t F' \dif s , \\
        \theta_{x_2 x_2}  & =  \theta_{y y}  =  \theta \cdot \br{\int_0^t F' \dif s}^2 + \theta \cdot \int_0^t F''\dif s , \text{ and} \\
        \theta_t  & =  -\theta \sigma_t^2 + 2 \theta f(x_2 + y) - \gamma_t \sigma_t \theta_y .
    \end{align*}

    Using the general It\^o formula (\cref{thm:Ayed-Kuo_differential_formula}), we get
    \begin{align*}
        \dif \theta
        = &  \theta_t \dif t + \theta_{x_1} \dif X^{(1)}_t + \theta_{x_2} \dif X^{(2)}_t + \theta_{y} \dif Y^{(t)}  \\
          &  + \frac12 \theta_{x_1 x_1} \br{\dif X^{(1)}_t}^2 + \frac12 \theta_{x_2 x_2} \br{\dif X^{(2)}_t}^2 + \theta_{x_1 x_2} \dif X^{(1)}_t \dif X^{(2)}_t - \frac12 \theta_{y y} (\dif Y^{(t)})^2  \\
        = &  \theta_t \dif t + 2 \theta \sigma_t \dif W_t + \bcancel{\theta_{x_2} \gamma_t \dif W_t} - \bcancel{\theta \gamma_t \dif W_t}  \\
          &  + 2 \theta \sigma_t^2 \dif t + \cancel{\frac12 \theta_{x_2 x_2} \gamma_t^2 \dif t} + 2 \theta_y \gamma_t \sigma_t \dif t - \cancel{\frac12 \theta_{y y} \gamma_t^2 \dif t}  \\
        = &  \br{\theta_t + 2 \theta \sigma_t^2 + 2 \theta_y \gamma_t \sigma_t} \dif t + 2 \theta \sigma_t \dif W_t  \\
        = &  \bs{\theta \sigma_t + 2 \theta f(x_2 + y) + 2 \gamma_t \sigma_t \theta \int_0^t F' \dif s} \dif t + 2 \theta \sigma_t \dif W_t .
    \end{align*}
    Finally, using \( V_t = \theta\br{t, X^{(1)}_t, X^{(2)}_t, Y^{(t)}} \), we get the stochastic differential equation.
\end{proof}

\subsection{A novel braiding technique for the Skorokhod sense}

In the prior section, we showed the existence of the solution via the Ayed--Kuo differential formula. However, the procedure started with intelligently guessing an ansatz for the solution and applying the differential formula to it. Can a solution be found without this \textquote{guessing}? In this section, we use elementary ideas from Malliavin calculus to interpret the stochastic differential equation in the Skorokhod sense. We introduce an iterative \textquote{braiding} technique in the spirit of Trotter's product formula\cite{Trotter1959} that allows us to construct the solution without needing to know the form of the solution. Note that we expect to arrive at the same solution as in \cref{sec:SDE_drift_Ayed-Kuo} since under the definition of the Ayed--Kuo integral using \( L^2(\Omega) \) convergence, the Hitsuda--Skorokhod integral and the Ayed--Kuo integrals are equivalent, as shown in \cite[theorem 2.3]{PeterParczewski2017}. In what follows, we briefly introduce some ideas of Malliavin calculus and Skorokhod integral so that we can introduce our braiding technique.

A well known extension of the It\^o integral is Hitsuda--Skorokhod integral. For this text, we shall introduce the Hitsuda--Skorokhod integral as the adjoint of the Gross--Malliavin derivative. Let us first set up the spaces to operate on. Consider the probability space \( (\Omega, \F, P)  \) where \( \F \) is the \( \sigma \)-field generated by the Brownian motion. Let \( \mathbb{H} = L^2[0, 1] \) be the space of square integrable functions defined on the positive reals. For any $h \in \mathbb{H}$, consider the Wiener integral
\begin{align*}
    W(h) = \int_0^1 h(t) \, dW_t.
\end{align*}
In particular, if $h = \mathbbm{1}_{[0, \frac12]} \in \mathbb{H}$ then
\begin{align*}
     W\br{\mathbbm{1}_{[0, \frac12]}} = \int_0^1 \mathbbm{1}_{[0, \frac12]}(t) \, dW_t = W_\frac12.
\end{align*}
This Hilbert space $\mathbb{H}$ plays an important role in the definition of the derivative. Let $\mathcal{S}$ be the class of smooth random variables such that $F \in \mathcal{S}$ has the form
\begin{align*}
    F = f\br{W(h_1),W(h_2)\dots,W(h_n)}, \quad h_i \in \mathbb{H} , \, i \in \{ 1, 2, \dotsc, n \},
\end{align*}
where $f$ is a real valued \( n \)-dimensional smooth function whose derivatives have at most polynomial growth.

\begin{definition}[{\cite[definition 1.2.1]{Nualart2006}}]
The Gross--Malliavin derivative of a smooth random variable $F \in \mathcal{S}$ is the real valued random variable given by
\begin{align*}
    D_t F = \sum_{i=1}^n \partial_i f\br{W(h_1),W(h_2)\dots,W(h_n)} h_i(t),
\end{align*}
where $ d_i$  is the derivative with respect to the $i$th variable.
\end{definition}

We denote $\mathbb{D}^{1,2}$ as the closure of the derivative operator $ D $ from $L^2(\Omega)$ to \( L^2(\Omega; \mathbb{H}) \). In other words, $\mathbb{D}^{1,2}$ is the completion of the class of smooth Brownian functionals with respect to the inner product
\begin{align*}
    \ba{F, G}_{1,2} = E\br{FG} +  E\br{\ba{DF, DG}_{\mathbb{H}}} .
\end{align*}


We now introduce the Skorokhod integral operator \( \delta \).
\begin{definition}[{\cite[definition 1.3.1]{Nualart2006}}]
    We denote by \( \delta \) the adjoint of the operator \( D \). That is, \( \delta \) is an unbounded operator on \( L^2(\Omega; \mathbb{H}) \) with values in \( L^2(\Omega) \) such that:
    \begin{enumerate}
        \item  The domain of \( \delta \) is the set of \( \mathbb{H} \)-valued square integrable random variables \( u \in L^2(\Omega; \mathbb{H}) \) such that
        for any \( F \in \mathbb{D}^{1,2} \), where \( c \) is some constant depending on \( u \).
        \[ \E\br{\ba{D F, u}_\mathbb{H}} \leq c \norm{F}_2 . \]
        \item  If u belongs to the domain of \( \delta \), then \( \delta \)(u) is the element of \( L^2(\Omega) \) characterized by
        \[ \E\br{F \delta u} = \E\br{\ba{D F, u}_\mathbb{H}} . \]
        for any \( F \in \mathbb{D}^{1,2} \).
    \end{enumerate}
\end{definition}

It is natural to ask about the nature of relationship of these two stochastic integral. While that is an open question, we refer to the following result.

\begin{theorem}[{\cite[theorem 2.3]{PeterParczewski2017}}]
Let $f$ be an adapted $L^2$-continuous stochastic process  and $\phi$ be an instantly independent  $L^2$-continuous stochastic process such that the sequence
\begin{align*}
    \sum^n_{i=1} f(t_{i-1}) \phi(t_i)  \br{W_{t_i} - W_{t_{i-1}}},
\end{align*}
converges strongly in $L^2(\Omega)$ as \( \norm{\Delta_n} \to 0 \). Then the limit $I(f\psi )$ equals the Hitsuda--Skorokhod integral $\delta(f\psi)$ in $Dom( \delta) $.
\end{theorem}

Now, we move on to finding the solution of the linear stochastic differential equation when the anticipating integral is taken in the sense of Skorokhod. First, fix the family of translation on the space of continuous functions starting at the origin in the Cameron--Martin direction given by
\begin{equation*}
    (A_t(\omega))_s  =  \omega_s - \int_0^{t \wedge s} \sigma(u) \dif u
    \qquad  \text{ and }  \qquad
    (T_t(\omega))_s  =  \omega_s + \int_0^{t \wedge s} \sigma(u) \dif u .
\end{equation*}

We look at an existence result for stochastic differential equations in the Skorokhod sense.
\begin{lemma}  \label{thm:SDE_Skorokhod_base}
    Suppose \( \sigma \in L^2[0, 1] \) and \( \xi \in L^p(\Omega) \) for some \( p > 2 \).
    Then the stochastic differential equation
    \begin{equation}  \label{eqn:SDE_Skorokhod_base}
        \left\{
        \begin{aligned}
            \dif Z_t  & =  \sigma(t) ~ Z_t \dif W_t  \\
            Z_0  & =  \xi ,
        \end{aligned}
        \right.
    \end{equation}
    has the unique solution given by
    \begin{equation}  \label{eqn:SDE_Skorokhod_base_solution}
        Z_t =  (\xi \circ A_t) ~ \mathcal{E}_t .
    \end{equation}
\end{lemma}
\begin{proof}
	It is clear that the family \( \bc{ (\xi \circ A_t) \mathcal{E}_t \given t \in [0, 1] } \) is \( L^r(\Omega) \)-bounded for all \( r < p \) by Girsanov's theorem and H\"older's inequality. Let \( G \) be any smooth random variable. Multiply both sides of \cref{eqn:SDE_Skorokhod_base} by \( G \). With the process \( X \) given by \eqref{eqn:SDE_Skorokhod_base_solution},
	\begin{align*}
		\E\br{G\int_0^t \sigma(s) ~ Z_s \dif W_s}
		& =  \E\br{\int_0^t \sigma(s) ~ Z_s ~ D_s G \dif s}  \\
		& =  \E\br{\xi \int_0^t \sigma(s) ~ (D_s G)(T_s) \dif s}  \qquad  \text{(using Girsanov theorem)}  \\
		& =  \E\br{\xi \int_0^t \frac{\dif}{\dif s} G(T_s) \dif s}  \\
		& =  \E\br{\xi(G(T_t) - G)}  \\
		& =  \E\br{\xi(A_t) ~ \mathcal{E}_t ~ G}  - \E\br{\xi ~ G}  \qquad  \text{(again by Girsanov theorem)}  \\
		& =  \E\br{Z_t ~ G} - \E\br{\xi ~ G} .
	\end{align*}
	Thus, a solution of the stochastic equation \cref{eqn:SDE_Skorokhod_base} is explicitly given by \eqref{eqn:SDE_Skorokhod_base_solution}.

	Uniqueness follows since the solution of \cref{eqn:SDE_Skorokhod_base} started at \( \xi \equiv 0 \) is identically zero at all times.
\end{proof}

Now we introduce the braiding technique to solve \cref{eqn:SDE_drift}, where \( \gamma \in L^2[0, 1] \) and \( f: \R \to \R \). To simplify notation, define
\begin{align*}
    I_\gamma  & =  \int_0^1 \gamma_s \dif W_s , \\
    A_u^v(\omega_\cdot)  & =  \omega_\cdot - \int_u^{(\cdot \wedge v) \vee u} \sigma(s) \dif s , \\
    E_u^v  & =  \exp\bs{\int_u^v \sigma(s) \dif W_s - \frac12 \int_u^v \sigma(s)^2 \dif s} , \text{ and} \\
    g_u^v  & =  \exp\bs{f(I_\gamma) ~ (v - u)} .
\end{align*}
Directly from the definitions above, for any \( u < v < w \), we get the compositions
\begin{align*}
    A_v^w \circ A_u^v  & =  A_u^w , \\
    E_u^v \circ A_v^w  & =  E_u^v , \\
    g_u^v \circ A_v^w  & =  \exp\bs{f(I_\gamma \circ A_v^w) ~ (v - u)} ,
\end{align*}
and the products
\begin{align*}
    E_u^v \cdot E_v^w  & =  E_u^w , \text{ and} \\
    g_u^v \cdot g_v^w  & =  g_u^w .
\end{align*}
We suppress the dependence on \( \omega \) for notational convenience.

Fix \( t \in [0, 1] \), and consider a sequence of partitions \( \Delta_n = \bc{0 = t_0 < t_1 < \dotsb < t_n = t} \) of \( [0, t] \) such that \( \norm{\Delta_n} = \sup\bc{t_i - t_{i-1} \given i \in [n]} \to 0 \). On each subinterval, we
\begin{enumerate}
    \item  \label{itm:diffusion_step}  solve the equation having only the diffusion with the initial condition as the solution of the previous step, and
    \item  \label{itm:drift_step}  use the solution obtained in step \ref{itm:diffusion_step} as the initial condition and solve the equation having only the drift.
\end{enumerate}
For the first subinterval, the initial condition of step \ref{itm:diffusion_step} is taken to be \( \xi \). For a visual representation of the idea, see \cref{fig:braiding}.

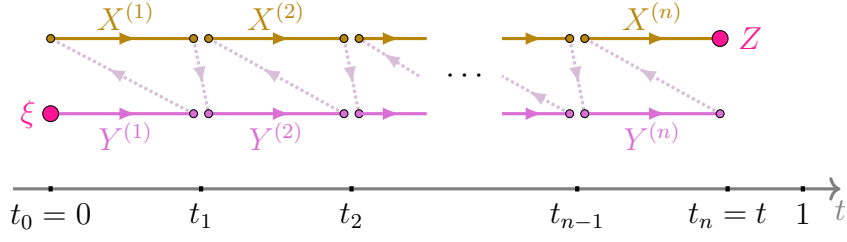
\begin{figure}
    \centering
    \begin{tikzpicture}[
        decoration={markings,
        mark=at position 0.6 with {\arrow{latex}}}
    ]
        \draw [->, very thick, color=gray] (-0.5,0) -- (10.5,0) node[below] {\( t \)};
        \foreach \x/\xtext in {0/t_0=0, 2/t_1, 4/t_{2}, 7/t_{n-1}, 9/t_n=t, 10/1}
            \draw [ultra thick] (\x cm,1pt) -- (\x cm,-1pt) node[anchor=north] {$\xtext$};

        \draw [very thick, densely dotted, color=Thistle, postaction={decorate}] (1.9,1) -- (0,2);
        \draw [very thick, densely dotted, color=Thistle, postaction={decorate}] (1.9,2) -- (2.1,1);
        \draw [very thick, densely dotted, color=Thistle, postaction={decorate}] (3.9,1) -- (2.1,2);
        \draw [very thick, densely dotted, color=Thistle, postaction={decorate}] (3.9,2) -- (4.1,1);
        \draw [very thick, densely dotted, color=Thistle, postaction={decorate}] (4.9,1.5) -- (4.1,2);
        \draw [very thick, densely dotted, color=Thistle, postaction={decorate}] (6.9,1) -- (6,1.5);
        \draw [very thick, densely dotted, color=Thistle, postaction={decorate}] (6.9,2) -- (7.1,1);
        \draw [very thick, densely dotted, color=Thistle, postaction={decorate}] (8.9,1) -- (7.1,2);

        \draw [very thick, color=Orchid, postaction={decorate}] (0,1) -- (1.9,1);
        \draw [color=Orchid] (1,0.7) node {\( Y^{(1)} \)};
        \draw [fill=Orchid] (1.9,1) circle (1.5pt);
        \draw [fill=DeepPink] (0,1) circle (3pt);
        \draw [color=DeepPink] (-0.3,1) node {\( \xi \)};

        \draw [very thick, color=DarkGoldenrod, postaction={decorate}] (0,2) -- (1.9,2);
        \draw [color=DarkGoldenrod] (1,2.3) node {\( X^{(1)} \)};
        \draw [fill=DarkGoldenrod] (0,2) circle (1.5pt);
        \draw [fill=DarkGoldenrod] (1.9,2) circle (1.5pt);

        \draw [very thick, color=Orchid, postaction={decorate}] (2.1,1) -- (3.9,1);
        \draw [color=Orchid] (3,0.7) node {\( Y^{(2)} \)};
        \draw [fill=Orchid] (2.1,1) circle (1.5pt);
        \draw [fill=Orchid] (3.9,1) circle (1.5pt);

        \draw [very thick, color=DarkGoldenrod, postaction={decorate}] (2.1,2) -- (3.9,2);
        \draw [color=DarkGoldenrod] (3,2.3) node {\( X^{(2)} \)};
        \draw [fill=DarkGoldenrod] (2.1,2) circle (1.5pt);
        \draw [fill=DarkGoldenrod] (3.9,2) circle (1.5pt);

        \draw [fill=Orchid] (4.1,1) circle (1.5pt);
        \draw [very thick, color=Orchid, postaction={decorate}] (4.1,1) -- (5,1);
        \draw [very thick, color=DarkGoldenrod,  postaction={decorate}] (4.1,2) -- (5,2);
        \draw [fill=DarkGoldenrod] (4.1,2) circle (1.5pt);

        \draw [very thick, color=Orchid, postaction={decorate}] (6,1) -- (6.9,1);
        \draw [very thick, color=DarkGoldenrod, postaction={decorate}] (6,2) -- (6.9,2);
        \draw [fill=DarkGoldenrod] (6.9,2) circle (1.5pt);
        \draw [fill=Orchid] (6.9,1) circle (1.5pt);

        \draw [very thick, color=Orchid, postaction={decorate}] (7.1,1) -- (8.9,1);
        \draw [color=Orchid] (8,0.7) node {\( Y^{(n)} \)};
        \draw [fill=Orchid] (7.1,1) circle (1.5pt);
        \draw [fill=Orchid] (8.9,1) circle (1.5pt);

        \draw [very thick, color=DarkGoldenrod, postaction={decorate}] (7.1,2) -- (8.9,2);
        \draw [color=DarkGoldenrod] (8,2.3) node {\( X^{(n)} \)};
        \draw [fill=DarkGoldenrod] (7.1,2) circle (1.5pt);
        \draw [fill=DeepPink] (8.9,2) circle (3pt);
        \draw [color=DeepPink] (9.3,2) node {\( Z \)};

        \draw[color=Black] (5.5,1.5) node  {\( \dotsb \)};
    \end{tikzpicture}
    \caption{A graphical representation of the braiding technique.}
    \label{fig:braiding}
\end{figure}

We explicitly demonstrate the process for the first two subintervals. An index \( (i) \) in the superscript refers to the \( i \)th subinterval.

\paragraph{First subinterval.}
\begin{enumerate}
    \item  The stochastic differential equation that we want to solve is
    \begin{equation*}
        \left\{
        \begin{aligned}
            \dif Y^{(1)}_u  & =  \sigma(u) Y^{(1)}_u \dif W_u ,  \quad  u \in [0, t_1] , \\
                 Y^{(1)}_0  & =  \xi .
        \end{aligned}
        \right.
    \end{equation*}
    \Cref{thm:SDE_Skorokhod_base} gave us the almost sure unique solution \( Y^{(1)}_u = (\xi \circ A_0^u) ~ E_0^u \), so
    \[ Y^{(1)}_{t_1} = (\xi \circ A_0^{t_1}) ~ E_0^{t_1} \]
    on a set \( \Omega_1 \), where \( \Pr\br{\Omega_1} = 1 \).

    \item  For each \( \omega \in \Omega_1 \), we want to solve the ordinary differential equation
    \begin{equation*}
        \left\{
        \begin{aligned}
            \dif X^{(1)}_u  & =  f(I_\gamma) ~ X^{(1)}_u \dif u ,  \quad  u \in [0, t_1] , \\
                 X^{(1)}_0  & =  Y^{(1)}_{t_1} .
        \end{aligned}
        \right.
    \end{equation*}
    By the existence and uniqueness theorem of ordinary differential equations, the unique solution is given by \( X^{(1)}_u  =  Y^{(1)}_{t_1} ~ g_0^u  =  (\xi \circ A_0^{t_1}) ~ E_0^{t_1} ~ g_0^u \), and so
    \[ X^{(1)}_{t_1}  =  (\xi \circ A_0^{t_1}) ~ E_0^{t_1} ~ g_0^{t_1}  . \]
\end{enumerate}

\paragraph{Second subinterval.}
\begin{enumerate}
    \item  The stochastic differential equation that we want to solve is
    \begin{equation*}
        \left\{
        \begin{aligned}
            \dif Y^{(2)}_u  & =  \sigma(u) ~ Y^{(2)}_u \dif W_u ,  \quad  u \in [t_1, t_2] , \\
                 Y^{(2)}_{t_1}  & =  X^{(1)}_{t_1} .
        \end{aligned}
        \right.
    \end{equation*}
    \Cref{thm:SDE_Skorokhod_base} gives us the almost sure unique solution \( Y^{(2)}_u = (X^{(1)}_{t_1} \circ A_{t_1}^u) ~ E_{t_1}^u \). Now,
    \begin{align*}
        Y^{(2)}_u
        & =  \bs{\br{(\xi \circ A_0^{t_1}) ~ E_0^{t_1} ~ g_0^{t_1}} \circ A_{t_1}^u} ~ E_{t_1}^u  \\
        & =  (\xi \circ A_0^{t_1} \circ A_{t_1}^u) ~ E_0^{t_1} ~ E_{t_1}^u ~ (g_0^{t_1} \circ A_{t_1}^u)  \\
        & =  (\xi \circ A_0^u) ~ E_0^u ~ (g_0^{t_1} \circ A_{t_1}^u) ,
    \end{align*}
    where we used the fact that \( E_0^{t_1} \) is invariant under \( A_{t_1}^u \). This is because, by definition,
    \[ A_{t_1}^u(\omega_\cdot)  =  \omega_\cdot - \int_{t_1}^{(\cdot \wedge u) \vee t_1} \sigma(s) \dif s . \]
    Now, for \( E_0^{t_1} \), we have \( t \in [0, t_1] \). Therefore,
    \[ A_{t_1}^u(\omega_t)  =  \omega_t - \int_{t_1}^{t_1} \sigma(s) \dif s = \omega_t , \]
    showing the invariance. This gives the motivation behind why we define \( A \) as such, and is a key trick in the method.

    Continuing, we get
    \[ Y^{(2)}_{t_2} = (\xi \circ A_0^{t_2}) ~ E_0^{t_2} ~ (g_0^{t_1} \circ A_{t_1}^{t_2}) \]
    on a set \( \Omega_2 \subseteq \Omega_1 \), where \( \Pr\br{\Omega_2} = 1 \).

    \item  For each \( \omega \in \Omega_2 \), we have the ordinary differential equation
    \begin{equation*}
        \left\{
        \begin{aligned}
            \dif X^{(2)}_u  & =  f(I_\gamma) ~ X^{(2)}_u \dif u ,  \quad  u \in [t_1, t_2] , \\
                 X^{(2)}_{t_1}  & =  Y^{(2)}_{t_1} .
        \end{aligned}
        \right.
    \end{equation*}
    The unique solution is given by \( X^{(2)}_u  =  Y^{(2)}_{t_1} ~ g_{t_1}^u \). Using the definition of \( Y^{(2)}_{t_1} \) and the fact that \( A_{t_2}^{t_2} \) is the identity function,
    \begin{align*}
        X^{(2)}_{t_2}
        & =  \bs{(\xi \circ A_0^{t_2}) ~ E_0^{t_2} ~ (g_0^{t_1} \circ A_{t_1}^{t_2})} ~ (g_{t_1}^{t_2} \circ A_{t_2}^{t_2})  \\
        & =  (\xi \circ A_0^{t_2}) ~ E_0^{t_2} ~ \prod_{i = 1}^2 (g_{t_1}^{t_2} \circ A_{t_2}^{t_2})
    \end{align*}
\end{enumerate}

It should now become obvious what the pattern is. We prove this using induction in the following lemma.

\begin{lemma}  \label{thm:braiding}
    Let \( \xi \in L^p(\Omega) \) for some \( p > 2 \). Consider the \( k \)th subinterval \( u \in [t_{k-1}, t_k] \) for any \( k \in [n] \), and define
    \begin{enumerate}
        \item  the stochastic differential equation
        \begin{equation*}
            \left\{
            \begin{aligned}
                   \dif Y^{(k)}_u  & =  \sigma(u) ~ Y^{(k)}_u \dif W_u ,  \quad  u \in [t_{k-1}, t_k] , \\
                Y^{(k)}_{t_{k-1}}  & =  X^{({k-1})}_{t_{k-1}} , \text{ and}
            \end{aligned}
            \right.
        \end{equation*}

        \item  the ordinary differential equation
        \begin{equation*}
            \left\{
            \begin{aligned}
                   \dif X^{(k)}_u  & =  f(I_\gamma) ~ X^{(k)}_u \dif u ,  \quad  u \in [t_{k-1}, t_k] , \\
                X^{(k)}_{t_{k-1}}  & =  Y^{(k)}_{t_k} .
            \end{aligned}
            \right.
        \end{equation*}
    \end{enumerate}
    Then there exists a set \( \Omega_k \subseteq \Omega \) with \( \Pr\br{\Omega_k} = 1 \) such that on \( \Omega_k \), we have
    \[ X^{(k)}_{t_k} =  (\xi \circ A_0^{t_k}) ~ E_0^{t_k} ~ \prod_{i = 1}^k (g_{t_{i-1}}^{t_i} \circ A_{t_i}^{t_k}) . \]
\end{lemma}

\begin{proof}
    \emph{Base cases.} This is true for \( k = 1 \) and \( k = 2 \) as shown in the computations above.

    \emph{Induction step.} Assume that the result holds for \( k = m-1 \). This means that there exists \( \Omega_{m-1} \) with \( \Pr\br{\Omega_{m-1}} = 1 \) such that on \( \Omega_{m-1} \), we have
    \[ X^{(m-1)}_{t_{m-1}} =  (\xi \circ A_0^{t_{m-1}}) ~ E_0^{t_{m-1}} \cdot \prod_{i = 1}^{m-1} (g_{t_{i-1}}^{t_i} \circ A_{t_i}^{t_{m-1}}) . \]

    Using the ideas of computations on the second subinterval, we get that there exists \( \Omega_m \) with \( \Pr\br{\Omega_m} = 1 \) such that on \( \Omega_m \), we have
    \[ Y^{(m)}_{t_m} = (\xi \circ A_0^{t_m}) ~ E_0^{t_m} ~ \prod_{i = 1}^{m-1} (g_{t_{i-1}}^{t_i} \circ A_{t_i}^{t_m}) . \]
    Since \( A_{t_m}^{t_m} \) is the identity function, on \( \Omega_m \), we have
    \begin{align*}
        X^{(m)}_{t_m}
        & =  Y^{(m)}_{t_{m-1}} ~ g_{t_{m-1}}^{t_m}  \\
        & = (\xi \circ A_0^{t_m}) ~ E_0^{t_m} ~ \prod_{i = 1}^m (g_{t_{i-1}}^{t_i} \circ A_{t_i}^{t_m}) .
    \end{align*}

    The proof is now complete by mathematical induction.
\end{proof}

We are now able to derive a closed form solution of \cref{eqn:SDE_drift} in the Skorokhod sense. This is the main theorem of the section.

\begin{theorem}  \label{thm:SDE_drift_Skorokhod}
    Suppose \( \sigma, \gamma \in L^2[0, 1] \), \( f: \R \to \R \), and \( \xi \in L^p(\Omega) \) for some \( p > 2 \). Then the unique solution of \cref{eqn:SDE_drift} in the Skorokhod sense is given by
    \begin{align*}  \label{eqn:SDE_drift_solution_Skorokhod}
        Z_t  =  (\xi \circ A_0^t) \exp
        &  \left[ \int_0^t \sigma(s) \dif W_s - \frac12 \int_0^t \sigma(s)^2 \dif s \right.  \\
        &  \left. + \int_0^t f\br{ \int_0^1 \gamma_u \dif W_u - \int_s^t \gamma_u ~ \sigma(u) \dif u } \dif s \right] .
    \end{align*}
\end{theorem}

\begin{remark}
    Note that \( \xi \) may depend on the Wiener process.
\end{remark}

\begin{proof}
    Using \cref{thm:braiding}, for any \( t \in [0, 1] \), we have
    \[ X^{(n)}_t =  (\xi \circ A_0^t) ~ E_0^t ~ \prod_{i = 1}^k (g_{t_{i-1}}^{t_i} \circ A_{t_i}^t) . \]
    Now,
    \begin{align*}
        \prod_{i = 1}^k (g_{t_{i-1}}^{t_i} \circ A_{t_i}^t)
        & =  \prod_{i = 1}^k \exp\bs{f(I_\gamma \circ A_{t_i}^t) ~ (t_i - t_{i-1})}  \\
        & =  \exp\bs{\sum_{i = 1}^k f\br{\int_0^1 \gamma_u \dif W_u - \int_0^t \gamma_u ~ \sigma(u) \dif u} \Del t_i} .
    \end{align*}
    Finally, taking \( n \to \infty \), we get
    \begin{align*}
        Z_t
        & =  \lim_{n \to \infty} X^{(n)}_t  \\
        & =  (\xi \circ A_0^t) ~ E_0^t ~ \exp\bs{\int_0^t f\br{\int_0^1 \gamma_u \dif W_u - \int_0^t \gamma_u ~ \sigma(u) \dif u} \dif s} ,
    \end{align*}
    which exactly equals the proposed solution.

    The solution exists almost surely, due to the continuity of the measure. Moreover, the solution is unique. For if not, there are two solutions which disagree for the first time on a particular interval, say the \( k \)th interval. Recall that the solutions obtained using Malliavin calculus and also for ordinary differential equations are unique for each interval of time. Therefore, such a disagreement would violate these uniqueness results.
\end{proof}

\section{Large deviation principles}  \label{sec:LDP}

The theory of large deviation allow us to find probabilities of rare events that decay exponentially. Our goal is to derive large deviation principles for the solutions of LSDEs that we derived in \cref{sec:SDE_solution}. But first, we give the formal setting for sample path large deviations.

\begin{definition}
Let \( (\mathcal{X}, d) \) be a Polish space and \( \br{\mu^\epsilon}_{\epsilon > 0} \) a sequence of Borel probability measures on \( \mathcal{X} \). Suppose \( I: \mathcal{X} \to \infty \) is a lower semicontinuous functional. Then the sequence \( \br{\mu^\epsilon}_{\epsilon > 0} \) is said to satisfy a \emph{large deviation principle}  on \( \mathcal{X} \) with \emph{rate function} \( I \) if and only if
    \begin{enumerate}
        \item  (upper bound)  for every closed set \( F \subseteq \mathcal{X} \),
            \[ \varlimsup_{\epsilon \to 0} \epsilon \log \mu^\epsilon(F)  \leq  - \inf_{x \in F} I(x) , \]
        \item  (lower bound)  and for every open set \( G \subseteq \mathcal{X} \),
            \[ \varliminf_{\epsilon \to 0} \epsilon \log \mu^\epsilon(G)  \geq  - \inf_{x \in G} I(x) . \]
    \end{enumerate}
\end{definition}

The next result states how large deviation principles are transferred under continuous transformations.
\begin{theorem}[{\cite[theorem 4.2.1]{DemboZeitouni1998}}] \label{thm:contraction}
Let $\mathcal{X}$ and $\mathcal{Y}$ be two polish spaces, $I$ a rate function on $\mathcal{X}$, and $f$ a continuous function mapping $\mathcal{X}$ to $\mathcal{Y}$. Then the following conclusions hold.

\begin{enumerate}
    \item For each $y \in \mathcal{Y}$,
    \begin{align*}
        J(y) = \inf \left\{I(x)\, \vert\, x \in f^{-1}(y) \right\}
    \end{align*}
    is a rate function on $\mathcal{Y}$,
    \item If $\bc{X_n}$ satisfies large deviation principle on $\mathcal{X}$ with rate function $I$, then $\bc{f\left(X_n\right)}$ satisfies large deviation principle on $\mathcal{Y}$ with rate function $J$.
    \end{enumerate}
\end{theorem}

When are large deviation principles are conserved? To answer this question, we introduce the idea of superexponential closeness.
\begin{definition}[{\cite[definition 4.2.10]{DemboZeitouni1998}}]
For $ \epsilon > 0 $, let $X^\epsilon$ and $Y^\epsilon$ be families of random variables on $(\Omega, \F, P)$ that take values in $\mathcal{X}$. Then the families $X^\epsilon$ and $Y^\epsilon$ (and their corresponding families of laws) are said to be superexponentially close if
\begin{align*}
    \varlimsup_{\epsilon \to 0} \epsilon \log P\left\{ d(X^\epsilon, Y^\epsilon) > \delta \right\} = - \infty .
\end{align*}
\end{definition}

The following theorem says that large deviation principles are preserved for superexponentially close families.
\begin{theorem}[{\cite[theorem 4.2.13]{DemboZeitouni1998}}]   \label{thm:large deviation principle_exponential_equivalence}
    Suppose $X^\epsilon$ and $Y^\epsilon$ be superexponentially close families of random variables on $(\Omega, \F, P)$. Then $X^\epsilon$ follows large deviation principle with rate function $ I $ if and only if $Y^\epsilon$ follows large deviation principle with the same rate function $ I $.
\end{theorem}

Finally, we give an example of large deviation principle. Consider the family of process \( \br{\sqrt{\epsilon} W}_{\epsilon > 0} \), where a Wiener process \( W \) is scaled down by a parameter \( \sqrt{\epsilon} \). As \( \epsilon \to 0 \), we have \( \sqrt{\epsilon} W \to 0 \) almost surely. But at what rate does the convergence happen? This is answered by Schilder's theorem.
\begin{theorem}[Schilder {\cite[theorem 5.2.3]{DemboZeitouni1998}}] \label{thm:Schilder}
    The sequence of probability measure $\{p_{\epsilon}\}$ as $\epsilon \rightarrow 0$ follows Large Deviation Principle on  $C_0\br{\bs{0,1}}$ with rate function $I(f)$ where
    \[
    I(f) = \left\{
    \begin{array}{ll}
    \frac{1}{2} \int_{0}^{1} \vert f'(t) \vert^2 dt & \quad\text{if $f \in H^1$} \\
    \infty & \quad \text{otherwise} .
    \end{array}
    \right.
    \]
    where $H^1 = \{f \in C_0\br{\bs{0,1}}\ \vert \, f \text{ is absolutely continuous and } f' \in L^2[0,1] \}$.
\end{theorem}

\subsection{LSDEs with constant initial conditions}

Suppose \( \sigma \) and \( \gamma \) are deterministic functions of bounded variation on \( [0, 1] \). Moreover, suppose \( f \in C^2(\R) \) is Lipschitz continuous along with \( f, f', f'' \in L^1(\R) \). For a fixed \( \kappa \in \R \), consider the family of linear stochastic differential equations with parameter \( \epsilon > 0 \) given by
\begin{equation}  \label{eqn:SDE_drift_large deviation principle_constant}
    \left\{
    \begin{aligned}
        \dif Z^\epsilon_\kappa(t)  & =  f\br{\sqrt{\epsilon} \int_0^1 \gamma_s \dif W_s} Z^\epsilon_\kappa(t) \dif t + \sqrt{\epsilon} \sigma(t) ~ Z^\epsilon_\kappa(t) \dif W_t  \\
             Z^\epsilon_\kappa(0)  & =  \kappa ,
    \end{aligned}
    \right.
\end{equation}
Using the results from \cref{sec:SDE_solution}, the unique solutions to \cref{eqn:SDE_drift_large deviation principle_constant} are given by
\begin{align}  \label{eqn:SDE_drift_solution_Ayed--Kuo_large deviation principle_constant}
    Z^\epsilon_\kappa(t) =  \kappa \exp
    &  \left[ \sqrt{\epsilon} \int_0^t \sigma(s) \dif W_s - \frac{\epsilon}{2} \int_0^t \sigma(s)^2 \dif s \right.  \nonumber \\
    &  \left. + \int_0^t f\br{ \sqrt{\epsilon} \int_0^1 \gamma_u \dif W_u - \epsilon \int_s^t \gamma_u ~ \sigma(u) \dif u } \dif s \right]
\end{align}

In order to use the continuity principle (\cref{thm:contraction}), we need the following lemma.

\begin{lemma}  \label{thm:theta_continuity}
    The function \( \theta: C_0 \to C_\kappa \) defined by
    \begin{align*}
        \theta(x)  =  \kappa \exp
        &  \left[ \int_0^t \sigma(s) \dif x(s) - \frac{\epsilon}{2} \int_0^t \sigma(s)^2 \dif s \right.  \\
        &  \left. + \int_0^t f\br{ \int_0^1 \gamma_u \dif x(u) - \epsilon \int_s^t \gamma_u ~ \sigma(u) \dif u } \dif s \right] ,
    \end{align*}
    is continuous in the topology induced by the canonical supremum norm.
\end{lemma}

\begin{proof}
    We can write
    \[ \theta(x)  =  \kappa \exp\bs{ \phi(x) - \frac{\epsilon}{2} \int_0^t \sigma_s^2 \dif s + \psi(x) } , \]
    where \( \phi, \psi: C_0 \to C_0 \) is given by
    \begin{align*}
        \phi(x)  & =  \int_0^t \sigma(s) \dif x(s)  =  \sigma(t) x(t) - \int_0^t x(s) \dif \sigma(s) , \text{ and} \\
        \psi(x)  & =  \int_0^t f\br{ \int_0^1 \gamma_u \dif x(u) - \epsilon \int_s^t \gamma_u ~ \sigma(u) \dif u } \dif s .
    \end{align*}
    Using integration by parts,
    \begin{align*}
        \phi(x)  & =  \sigma(t) x(t) - \int_0^t x(s) \dif \sigma(s) , \text{ and} \\
        \psi(x)  & =  \int_0^t f\br{ \gamma(1) x(1) - \int_0^1 x(s) \dif \gamma_s - \epsilon \int_s^t \gamma_u ~ \sigma(u) \dif u } \dif s .
    \end{align*}
    Since multiplication by \( \kappa \exp\br{- \frac{\epsilon}{2} \int_0^t \sigma_s^2 \dif s} \) and \( \exp \) are continuous transformations, continuity of \( \theta \) is guaranteed if we prove continuity of \( \phi \) and \( \psi \). This is what we show below. For \( \phi \), we have
    \begin{align*}
        \norm{\phi(x) - \phi(y)}_\infty
        & =  \norm{\br{\sigma(t) x(t) - \int_0^t x(s) \dif \sigma(s)} - \br{\sigma(t) y(t) - \int_0^t y(s) \dif \sigma(s)}}_\infty  \\
        & \leq  \norm{\sigma(t) \br{x(t) - y(t)}}_\infty + \norm{\int_0^t (x(s) - y(s)) \dif \sigma(s)}_\infty  \\
        & \leq  \norm{\sigma}_\infty \norm{x - y}_\infty  +  \abs{\sigma(t) - \sigma(0)} \norm{x - y}_\infty  \\
        & \leq  3 \norm{\sigma}_\infty \norm{x - y}_\infty ,
    \end{align*}
    so \( \phi \) is continuous.

    For \( \psi \), if \( L_f \) is the Lipschitz constant for \( f \), we get
    \begin{align*}
        \norm{\psi(x) - \psi(y)}_\infty
        \leq &  \left\lVert \int_0^t L_f \left[ \br{ \gamma(1) x(1) - \int_0^1 x(s) \dif \gamma_s - \cancel{\epsilon \int_s^t \gamma_u ~ \sigma_u \dif u} } \right. \right.  \\
        & \qquad  - \left. \left. \br{ \gamma(1) y(1) - \int_0^1 y(s) \dif \gamma_s - \cancel{\epsilon \int_s^t \gamma_u ~ \sigma_u \dif u} } \right] \dif s \right\rVert_\infty  \\
        \leq &  L_f \norm{ \int_0^t \br{\gamma(1) \br{x(1) - y(1)} - \int_0^1 \br{x(s) - y(s)} \dif \gamma_s} \dif s }_\infty  \\
        \leq &  L_f \br{\norm{\gamma}_\infty \norm{x - y}_\infty + \abs{\gamma(1) - \gamma(0)} \norm{x - y}}_\infty  \\
        = &  3 L_f \norm{\gamma}_\infty \norm{x - y}_\infty ,
    \end{align*}
    which proves the continuity of \( \psi \).
\end{proof}

The following result now follows directly from the continuity of \( \theta \) (\cref{thm:theta_continuity}), the continuity principle (\cref{thm:contraction}), and Schilder's theorem (\cref{thm:Schilder}).
\begin{theorem}  \label{thm:SDE_drift_large deviation principle_constant}
    The laws of the solutions \( Z^\epsilon_\kappa \) given by \cref{eqn:SDE_drift_solution_Ayed--Kuo_large deviation principle_constant} of the family of stochastic differential equations given by \cref{eqn:SDE_drift_large deviation principle_constant} follow a large deviation principle on \( \br{C_\kappa, \norm{\cdot}_\infty} \) with the rate function
    \begin{equation}  \label{eqn:SDE_drift_large deviation principle_rate}
        J(y) = \inf \bc{I \circ \inv{\theta} (y)} ,
    \end{equation}
    where \( \theta \) is as defined in \cref{thm:theta_continuity}, and \( I \) is the rate function given by \cref{thm:Schilder}.
\end{theorem}

\subsection{LSDEs with random initial conditions}

Is it necessary for the family of linear stochastic differential equations \cref{eqn:SDE_drift_large deviation principle_constant} to start at a constant point \( \kappa \in \R \) in order for it to have a large deviation principle? In this section, we generalize \cref{thm:SDE_drift_large deviation principle_constant} and show that we can derive a similar result under a stronger version of exponential equivalence and more restrictive conditions on the functions \( f \), \( \sigma \), and \( \gamma \).

Suppose \( \sigma \) and \( \gamma \) are deterministic functions of bounded variation on \( [0, 1] \). Moreover, suppose \( f \in C^2(\R) \) is Lipschitz continuous along with \( f, f', f'' \in L^1(\R) \). Consider the family of linear stochastic differential equations with parameter \( \epsilon > 0 \) given by
\begin{equation}  \label{eqn:SDE_drift_large deviation principle_random}
    \left\{
    \begin{aligned}
        \dif Z^\epsilon_\xi(t)  & =  f\br{\sqrt{\epsilon} \int_0^1 \gamma_s \dif W_s} Z^\epsilon_\xi(t) \dif t + \sqrt{\epsilon} \sigma_t ~ Z^\epsilon_\xi(t) \dif W_t  \\
             Z^\epsilon_\xi(0)  & =  \xi^\epsilon ,
    \end{aligned}
    \right.
\end{equation}
where \( \xi^\epsilon \)s are random variables independent of the Wiener process \( W \). For each \( \epsilon \), just as before, the unique solution to \cref{eqn:SDE_drift_large deviation principle_random} is given by
\begin{align}  \label{eqn:SDE_drift_solution_Ayed--Kuo_large deviation principle_random}
    Z^\epsilon_\xi(t) =  \xi^\epsilon \exp
    &  \left[ \sqrt{\epsilon} \int_0^t \sigma(s) \dif W_s - \frac{\epsilon}{2} \int_0^t \sigma(s)^2 \dif s \right.  \nonumber \\
    &  \left. + \int_0^t f\br{ \sqrt{\epsilon} \int_0^1 \gamma_u \dif W_u - \epsilon \int_s^t \gamma_u ~ \sigma(u) \dif u } \dif s \right]
\end{align}

We now state a more general large deviation principle.
\begin{theorem}
    Let \( \kappa \in \R \) and consider the family of random variables \( \xi^\epsilon \) such that the following hold
    \begin{equation}  \label{eqn:SDE_drift_large deviation principle_exponential_equivalence}
        \lim_{\epsilon \to 0} \epsilon \log \E\bs{\br{\xi^\epsilon - \kappa}^2} = -\infty .
    \end{equation}
    Moreover, assume that the functions \( f, f', \sigma, \gamma \) are all bounded. Then the laws of the solutions \( Z^\epsilon_\xi \) given by \cref{eqn:SDE_drift_solution_Ayed--Kuo_large deviation principle_random} of the family of stochastic differential equations given by \cref{eqn:SDE_drift_large deviation principle_random} follow a large deviation principle on \( \br{C_\kappa, \norm{\cdot}_\infty} \) with the rate function given by \cref{eqn:SDE_drift_large deviation principle_rate}, where \( \theta \) is as defined in \cref{thm:theta_continuity}, and \( I \) is the rate function given by \cref{thm:Schilder}.
\end{theorem}
\begin{proof}
    Let \( V^\epsilon = Z^\epsilon_\xi - Z^\epsilon_\kappa \). Then \( V^\epsilon \) satisfies the stochastic differential equation
    \begin{equation}  \label{eqn:SDE_drift_large deviation principle_difference}
        \left\{
        \begin{aligned}
            \dif V^\epsilon_t   & =  f\br{\sqrt{\epsilon} \int_0^1 \gamma_s \dif W_s} V^\epsilon_t  \dif t + \sqrt{\epsilon} \sigma_t V^\epsilon_t  \dif W_t  \\
                 V^\epsilon_0  & =  \xi^\epsilon - \kappa ,
        \end{aligned}
        \right.
    \end{equation}
    whose solution is given by
    \begin{align*}
        V^\epsilon_t   =  \br{\xi^\epsilon - \kappa} \exp
        &  \left[ \sqrt{\epsilon} \int_0^t \sigma_s \dif W_s - \frac{\epsilon}{2} \int_0^t \sigma_s^2 \dif s \right.  \\
        &  \left. + \int_0^t f\br{ \sqrt{\epsilon} \int_0^1 \gamma_u \dif W_u - \epsilon \int_s^t \gamma_u ~ \sigma_u \dif u } \dif s \right] .
    \end{align*}

    Let \( \phi(z) = \abs{z}^2 \) and let \( U^\epsilon = \phi(V^\epsilon) \). From \cref{thm:SDE_drift_solution_squared}, \( U^\epsilon \) satisfies the integral equation
    \begin{align*}
        U^\epsilon(t)  =
        &  \br{\xi^\epsilon - \kappa}^2 +  2 \sqrt{\epsilon} \int_0^t \sigma_s ~ U^\epsilon_s  \dif W_s  \\
        & +  \epsilon \int_0^t \sigma_s^2 ~ U^\epsilon_s  \dif s  +  f\br{\int_0^1 \sqrt{\epsilon} ~ \gamma_s \dif W_s} \int_0^t U^\epsilon_s  \dif s  \\
        & +  2 \epsilon \int_0^t \gamma_s ~ \sigma_s ~ U^\epsilon_s  \int_0^s f'\br{\int_0^1 \sqrt{\epsilon} ~ \gamma_v \dif W_v - \epsilon \int_u^s \gamma_v ~ \sigma_v \dif v} \dif u \dif s .
    \end{align*}

    Fix \( \delta > 0 \) and let \( \tau = \inf\bc{t \in [0, 1] : \abs{V^\epsilon_t } \geq \delta} \). Taking expectation of the stopped process \( U^\epsilon_{t \wedge \tau}  \), we get
    \begin{align*}
        &  \E\br{U^\epsilon_{t \wedge \tau} }  \\
        & =  \E\bs{\br{\xi^\epsilon - \kappa}^2}  +  2 \sqrt{\epsilon} \E\bs{\int_0^{t \wedge \tau} \sigma_s ~ U^\epsilon_{s \wedge \tau}  \dif W_s}  \\
        & \quad +  \epsilon \E\bs{\int_0^{t \wedge \tau} \sigma_s^2 ~ U^\epsilon_{s \wedge \tau}  \dif s}  +  \E\bs{f\br{\int_0^1 \sqrt{\epsilon} ~ \gamma_s \dif W_s} \int_0^{t \wedge \tau} U^\epsilon_{s \wedge \tau}  \dif s}  \\
        & \quad +  2 \epsilon \E\bs{\int_0^{t \wedge \tau} \gamma_s ~ \sigma_s ~ U^\epsilon_{s \wedge \tau}  \int_0^s f'\br{\int_0^1 \sqrt{\epsilon} ~ \gamma_v \dif W_v - \epsilon \int_u^s \gamma_v ~ \sigma_v \dif v} \dif u \dif s} .
    \end{align*}
    The second integral on the right-hand side is a near-martingales by \cref{thm:Ayed-Kuo_integral_near-martingale}. Suppose \( f, f', \sigma, \gamma \) are all bounded by some \( M > 1 \). Using non-negativity of \( U^\epsilon \) and the near-martingale optional stopping theorem (\cref{thm:optional_stopping_near-martingale_special}), we get
    \begin{align*}
        \E\br{U^\epsilon_{t \wedge \tau} }
        \leq &  \E\bs{\br{\xi^\epsilon - \kappa}^2}  +  0 \\
          & +  \epsilon M^2 \E\bs{\int_0^{t \wedge \tau} U^\epsilon_{s \wedge \tau}  \dif s}  +  M \E\bs{\int_0^{t \wedge \tau} U^\epsilon_{s \wedge \tau}  \dif s}  \\
          & +  2 \epsilon M^3 \E\bs{\int_0^{t \wedge \tau} U^\epsilon_{s \wedge \tau}  \dif s}  \\
        \leq &  \E\bs{\br{\xi^\epsilon - \kappa}^2} + \br{M + 2 \epsilon M^3} \E\bs{\int_0^{t \wedge \tau} U^\epsilon_{s \wedge \tau}  \dif s} .
    \end{align*}
    By Gronwall's inequality, we get
    \[ \E\br{U^\epsilon_\tau}  =  \E\br{U^\epsilon_{1 \wedge \tau} }  \leq  \E\bs{\br{\xi^\epsilon - \kappa}^2} e^{M + 2 \epsilon M^3} . \]

    Since \( \phi(z) \) is a monotonically increasing non-negative function in \( \abs{z} \), we use Markov's inequality to get
    \begin{equation*}
        \Pr\bc{\abs{V^\epsilon_\tau} \geq \delta}
        = \Pr\bc{\phi(V^\epsilon_\tau) \geq \phi(\delta)}
        \leq  \frac{\E\br{\phi(V^\epsilon_\tau)}}{\phi(\delta)}
        =  \frac{\E\br{U^\epsilon(\tau)}}{\delta^2}
        \leq  \frac{\E\bs{\br{\xi^\epsilon - \kappa}^2}}{\delta^2} e^{M + 2 \epsilon M^3} .
    \end{equation*}
    Taking \( \log \) and multiplying by \( \epsilon \), we get
    \[ \epsilon \log \Pr\bc{\abs{V^\epsilon_\tau} \geq \delta}  \leq  \epsilon \log \E\bs{\br{\xi^\epsilon - \kappa}^2} - 2 \epsilon \log \delta + \epsilon (M + 2 \epsilon M^3) . \]
    Finally, taking limit of \( \epsilon \to 0 \) and using \cref{eqn:SDE_drift_large deviation principle_exponential_equivalence},
    \[ \lim_{\epsilon \to 0} \epsilon \log \Pr\bc{\abs{V^\epsilon_\tau} \geq \delta}  =  -\infty . \]
    This result allows us to say that \( Z^\epsilon_\xi \) and \( Z^\epsilon_\kappa \) are exponentially equivalent. Since exponentially equivalent families have the same large deviation principle due to \cref{thm:large deviation principle_exponential_equivalence}, \( Z^\epsilon_\xi \) follows a large deviation principle with the same rate function given by \cref{eqn:SDE_drift_large deviation principle_rate}.
\end{proof}

\printbibliography

\end{document}